\documentclass[a4paper,11pt]{amsart} 
\usepackage{amssymb, amsmath}
\usepackage{amscd}
\numberwithin{equation}{section}
\usepackage{epsfig}
\usepackage{amsmath}
\usepackage{amsfonts,amssymb,amsopn}
\usepackage{amsthm}
\usepackage{hyperref}
\usepackage{verbatim}%to comment out
\usepackage{color}
\usepackage[color,all]{xy}

\newcommand{\cE}{{\mathcal E}}
\newcommand{\cI}{{\mathcal I}}
\newcommand{\cL}{{\mathcal L}}
\newcommand{\cO}{{\mathcal O}}
\newcommand{\cP}{{\mathcal P}}
\newcommand{\C}{{\mathbb C}}
\newcommand{\K}{{\mathbb K}}
\newcommand{\bN}{\overline{N}}

\newcommand{\PP}{{\mathbb P}}
\newcommand{\tW}{\widetilde{E}}

\newcommand{\Ker}{\mathrm{Ker}}
\newcommand{\Image}{\mathrm{Im}}
\newcommand{\Iden}{\mathrm{Id}}
\newcommand{\Hom}{\mathrm{Hom}}
\newcommand{\rank}{\mathrm{rk}\,}
\newcommand{\Pic}{\mathrm{Pic}}
\newcommand{\ev}{\mathrm{ev}}

\newcommand{\Kc}{{K}}
\newcommand{\Oc}{{{\cO}_{C}}}
\newcommand{\sRat}{\underline{\mathrm{Rat}}\,}
\newcommand{\Rat}{\mathrm{Rat}\,}
\newcommand{\sPrin}{\mathrm{\underline{Prin}}\,}
\newcommand{\Prin}{\mathrm{Prin}\,}
\newcommand{\Gr}{\mathrm{Gr}}
\newcommand{\Quot}{\mathrm{Quot}}
\newcommand{\Supp}{\mathrm{Supp}}
\newcommand{\isom}{\xrightarrow{\sim}}
\newcommand{\Osc}{\mathrm{Osc}}

\newcommand{\codim}{\mathrm{codim}}

\newcommand{\GL}{\mathrm{GL}}

\newtheorem{theorem}{{\textbf Theorem}}[section]
\newtheorem{proposition}[theorem]{{\textbf Proposition}}
\newtheorem{corollary}[theorem]{{\textbf Corollary}}
\newtheorem{lemma}[theorem]{{\textbf Lemma}}
\newtheorem{ctn}[theorem]{{\textbf Caution}}
\newenvironment{caution}{\begin{ctn}\rm}{\end{ctn}}

\newtheorem{defn}[theorem]{{\textbf Definition}}
\newtheorem{notn}[theorem]{{\textbf Notation}}
\newtheorem{remit}[theorem]{{\textbf Remark}}
\newenvironment{remark}{\begin{remit}\rm}{\end{remit}}
\newenvironment{definition}{\begin{defn}\rm}{\end{defn}}
\newenvironment{notation}{\begin{notn}\rm}{\end{notn}}

\title[Quot schemes, Segre invariants, and inflectional loci]{Quot schemes, Segre invariants, and inflectional loci of scrolls over curves}

\author{George H. Hitching}

\address{Oslo Metropolitan University, Postboks 4, St. Olavs plass, 0130 Oslo, Norway.}

\email{gehahi@oslomet.no}

\begin{document}

\begin{abstract} Let $E$ be a vector bundle over a smooth curve $C$, and $S = \PP E$ the associated projective bundle. We describe the inflectional loci of certain projective models $\psi \colon S \dashrightarrow \PP^n$ in terms of Quot schemes of $E$. This gives a geometric characterisation of the Segre invariant $s_1 (E)$, which leads to new geometric criteria for semistability and cohomological stability of bundles over $C$. We also use these ideas to show that for general enough $S$ and $\psi$, the inflectional loci are all of the expected dimension. An auxiliary result, valid for a general subvariety of $\PP^n$, is that under mild hypotheses, the inflectional loci associated to a projection from a general centre are of the expected dimension. \end{abstract}

\maketitle

\section{Introduction}

Let $C$ be a smooth projective curve. It has long been known that the extrinsic geometry of maps of $C$ to projective space is closely connected with the cohomological properties of line bundles over $C$, and the geometry of the Riemann theta divisor and other Brill--Noether loci $W^r_d (C)$. See \cite{ACGH} for an overview, and \cite{KS}, \cite{CS} for further examples.

Now suppose $E \to C$ is a vector bundle of rank $r \ge 2$, and write $S := \PP E$. It is natural to ask again how the properties of $E$ and the moduli spaces containing it are reflected in the geometry of $S$ and its projective models. Of particular interest are properties which have no counterpart for line bundles, such as stability, and more generally Quot schemes of quotients of positive rank.

We mention some examples of results of this type. For $r = 2$, the natural identification between sections of the ruled surface $\PP E$ and line subbundles of $E$ is exploited in \cite{CCFMBrNoe} to study generalised Brill--Noether loci, and in \cite{CCFMnonsp} to study Hilbert schemes of scrolls and curves. Moving to higher rank; in \cite{Bri}, certain properties of the generalised theta divisor of $E$ are shown to depend on a variety of defective secants to $S \to |\cO_{\PP E} (1)|^*$. More closely related to the present work is \cite[{\S} 1]{IT}, where the degrees of minimal rank one quotients of $E^*$ are linked to ampleness of $S \dashrightarrow |\cO_{\PP E} (1)|^*$.

%\subsection*{Inflectional loci of scrolls} %A topic in projective geometry which has received considerable attention is the \emph{osculatory behaviour} of scrolls, and the related notions of inflectional loci and dual varieties. %Inflectional loci are singularities of higher differentials of maps $\PP E \dashrightarrow \PP^n$. %points where the image is ``flat'' to some order. 
%Scrolls over higher dimensional base; quadric fibrations; higher Gauss maps; enveloping bundles ?
 In the present article, we study a connection between \emph{inflectional properties} of linearly normal models (not necessarily embeddings) of $S$ and the Quot schemes and stability properties of the associated bundles $E$. Osculatory behaviour and inflectional properties of scrolls have been much studied. They are used to classify scrolls over $\PP^1$ in \cite{PT} and (via dual varieties) in \cite{PS}. The osculating spaces and dual varieties of elliptic scrolls are studied in \cite{MP}. In \cite{LMP} a formula is given which enumerates the inflection points of a scroll with suitable numerical invariants over a curve of any genus, when this is finite. (More generally, this formula gives the cohomology class of the inflectional locus, when this is of the expected dimension.)

To state our results, we firstly review the notions of stability and Segre invariants for bundles over curves. Recall that the \textsl{slope} of a vector bundle $F \to C$ is the ratio $\deg(F) / \rank (F)$. A bundle $E$ is said to be \textsl{stable} (resp., \textsl{semistable}) if $\mu(F) < \mu(E)$ (resp., $\mu(F) \le \mu (E)$) for all proper subbundles $F \subset E$. It is very well known that this property is of fundamental importance in moduli questions (see for example \cite{LeP}).

The notion of stability can be refined as follows. Suppose $E$ has rank $r$ and degree $d$. For $1 \le n \le r-1$, the \textsl{Segre invariant} $s_n (E)$ is defined by
\begin{equation} s_n (E) \ = \ \min \{ nd - r \cdot \deg (F) : F \subset E \hbox{ a vector subbundle of rank } n \} . \label{defnSegreInv} \end{equation}
The term ``invariant'' is used because $s_n ( E ) = s_n ( E \otimes L )$ for any line bundle $L$. Clearly $E$ is stable (resp., semistable) if and only if $s_n (E) > 0$ (resp., $s_n (E) \ge 0$) for $1 \le n \le r-1$. Segre invariants define stratifications on the moduli spaces of bundles over $C$, which are studied in \cite{BPL} and \cite{RTiB}. In \cite{LN}, the Segre invariants of rank two bundles are interpreted geometrically in terms of secant varieties to a projective model of the curve. This interpretation is generalised to higher rank and to symplectic and orthogonal bundles in \cite{CH1, CH2, CH5} and elsewhere.

In the present work, we study a link of a different kind between $s_1 (E)$ and the extrinsic geometry of $\PP E =: S$. Let us give an overview of our results. Let $\pi \colon S \to C$ be the projection. For any $M \in \Pic^0 (C)$, write $\cL_M$ for the line bundle $\cO_{\PP E} (1) \otimes \pi^* M$ over $S$.

In Proposition \ref{ParamInfl}, we give a key technical result, which is a criterion for the nonemptiness of the inflectional loci associated to the map $S \dashrightarrow | \cL_M |^*$ in terms of certain Quot schemes of $E$. This leads to the following characterisation of the Segre invariant $s_1 (E)$.

\vspace{0.2cm}

\noindent \textbf{Theorem \ref{MainA}.} \textit{Suppose $k \ge 0$, and let $E$ be a bundle of rank $r$ and degree $d$. Then the following are equivalent.
\begin{enumerate}
\item[(1)] $s_1 ( E ) > d + r ( 2g - 1 + k )$.
\item[(2)] For all $M \in \Pic^0 ( C )$ and all $x \in S$, the osculating space $\Osc^k ( S, x ) \subseteq | \cL_M |^*$ has dimension $kr$. In particular, the $k$th inflectional locus associated to the map $S \to | \cO_{\PP E} (1) \otimes \pi^* M |^*$ is empty.
\end{enumerate}}

\vspace{0.2cm}

\noindent Using this, we characterise the semistability property as follows.

\vspace{0.2cm}

\noindent \textbf{Theorem \ref{MainB}.} \textit{Let $E$ be a bundle of rank $r$ and slope $\mu < 1 - 2g$ over $C$. Then the following are equivalent.
\begin{enumerate}
\item[(1)] $E$ is semistable.
\item[(2)] For $1 \le n \le r-1$ and for $0 \le k < n \cdot \mu ( E^*) - (2g-1)$, the osculating space
\[ \Osc^k \left( \PP (\wedge^n E) , y \right) \ \subseteq \ | \cO_{\PP (\wedge^n E)}(1) \otimes \pi^* M |^* \]
is of the expected dimension $k \cdot \binom{r}{n}$ for all $y \in \PP ( \wedge^n E )$ and all $M \in \Pic^0 ( C )$.
\end{enumerate}}

\noindent Moreover, in {\S} \ref{CohomSt} we recall the notion of \emph{cohomological stability} introduced in \cite{EL}. In the sense of Theorem \ref{MainB}, cohomological stability is reflected more naturally than slope stability in the properties of the inflectional loci. This will be a subject of further study.

In {\S} \ref{dimInfl}, we give another application of the link given in Proposition \ref{ParamInfl} between inflectional loci and Quot schemes. Using familiar facts about Quot schemes of general bundles, we show that for a general scroll $S$ and general $M \in \Pic^0 ( C )$, the inflectional loci of $S \dashrightarrow |\cL_M|^*$ are all of the expected dimension (Theorem \ref{MainC}). The result proven in the appendix then shows the same is true for a projection of $S$ from a general centre in $|\cL_M|^*$ (Corollary \ref{IncompleteSystemsScrolls}). In particular, this confirms that the hypothesis of expected dimension in \cite[Theorem 2 and Corollary 1]{LMP} required for enumerating inflection points is satisfied when the parameters are chosen generally. Kleiman's transversality theorem \cite{Kle} is essential to the argument.

The plan of the paper is as follows. In {\S} \ref{OscIntro}, we recall background material on osculating spaces and inflectional loci of projective bundles. In {\S} \ref{SectionOscPpts} we obtain a description of the osculating spaces of $S \dashrightarrow |\cL_M|^*$ in terms of bundle-valued principal parts, which will be convenient for proofs. We can then prove Proposition \ref{ParamInfl}, which is the basis for Theorems \ref{MainA}, \ref{MainB} and \ref{MainC}.

In the appendix, we prove an auxiliary result required in {\S} \ref{dimInfl} which may be of independent interest: If $X$ is a smooth variety with a map to $\PP^n$ then, under a mild technical hypothesis, the inflectional loci behave as expected under general projections (Theorem \ref{IncompleteSystems}). This generalises a result in \cite{Pi1977} for curves, and also relies on Kleiman's theorem.

\subsection*{Acknowledgements} I thank Michael Hoff and Ragni Piene for helpful comments and discussions. I acknowledge gratefully a period of research leave supported by Oslo Metropolitan University in 2017--2018.

\subsection*{Notation} We work over an algebraically closed field $\K$ of characteristic zero. Throughout, $C$ denotes a projective smooth curve of genus $g \ge 1$, and $\Kc = T_C^*$ the canonical bundle of $C$. If $F \to C$ is a fibration, we write $F|_p$ for the fibre of $F$ at $p \in C$. If $D$ is a divisor on $C$, we abbreviate $V \otimes \Oc(D)$ to $V(D)$. If $E \to C$ is a vector bundle, we will occasionally abuse language by referring to a projective model $\psi \colon \PP E \dashrightarrow \PP^n$ as a ``scroll'' even though $\psi$ may not be an embedding.

\section{Osculating spaces and inflectional loci} \label{OscIntro}

In this section we recall basic definitions and facts, referring to \cite{LMP} and \cite{LM} for more detail. (Note that in these papers, ``$\PP V$'' denotes the projective space of hyperplanes in $V$, which we denote by $\PP V^*$ here.)

\subsection{Osculating spaces} \label{BackgroundOscSp} Let $X$ be a smooth projective variety and $\cL \to X$ a line bundle with nonempty linear system. Let $V \subseteq H^0 ( X, \cL )$ be a nonzero subspace of dimension $n+1$, and $\psi \colon X \dashrightarrow \PP V^* = \PP^n$ the natural map. For $k \ge 0$, we have the jet bundle $\cP^k (\cL)$, and the jet map
\[ j^k \colon \cO_X \otimes V \ \to \ \cP^k (\cL) \]
which sends a section of $\cL$ to its value modulo $\cI^{k+1}_x$ at each $x \in X$. This may be thought of as a truncated Taylor expansion.

\begin{definition} \label{DefnOsc} For $x \in X$ and $k \ge 0$, the \textsl{$k$th osculating space} $\Osc^k (X, x)$ is defined as
\[ \PP \Image \left( \left( j^k_x \right)^* \colon \cP^k(\cL)|_x^* \to V^* \right) \ \subseteq \ \PP V^* \ = \ \PP^n . \] %= \ \PP \Ker \left( V^* \to \left(V \cap H^0 ( X, \cL \otimes \cI_x^{k+1} ) \right)^* \right) . \]
Since $V \cap H^0 \left( X, \cL \otimes \cI_x^{k+1} \right) = \Ker ( j^k_x )$, we have also
\begin{equation} \Osc^k ( X, x ) %\ = \ \PP \Image \left( j^k_x \right)^* 
 \ = \ \PP \left( V \cap H^0 ( X, \cL \otimes \cI_x^{k+1} ) \right)^\perp. \label{DefnOscIm} \end{equation}
\noindent Clearly $\dim ( \Osc^k ( X, x ) )$ is lower semicontinuous in $x$. We write $d^k$ for the generic value of $\dim ( \Osc^k (X, x))$. \end{definition}

\begin{remark} \quad \label{DifferentialOperators} As $\Osc^0 ( X, x )$ is generated by the evaluation map $\ev_x \colon V \to \cL|_x$, it is simply the point $\psi(x)$. Furthermore, $\Osc^1 ( X, x )$ is the embedded tangent space of $\psi(X)$ at $\psi(x)$. In general, a section $s$ of $\cL$ vanishes to order at least $k$ at $x$ if and only if every differential operator of order at most $k$ at $x$ annihilates a local expression for $s$. Thus, by (\ref{DefnOscIm}) we see that $\Osc^k (X, x)$ is the subspace of $\PP H^0 (X, \cL)^*$ spanned by differential operators of order at most $k$ at $x$. Therefore,
\[ \dim \left( \Osc^k (X, x) \right) \ \le \ d^k \ \le \ \binom{\dim (X) + k}{\dim (X)} - 1 . \]
Note that the inequality on the right can be strict at the generic point; for example, if $X$ is a scroll as described in {\S} \ref{scrolls}, or a quadric fibration. \end{remark}

\begin{remark} If $\K = \C$ and we work with the classical topology, then \cite[II.3]{Pohl} gives another description of $\Osc^k (X, x)$. For each arc $A \subseteq X$ through $x$, the osculating space $\Osc^k ( A, x )$ is the limit of the secants spanned by $k$ distinct points of $A$ as these points approach $x$. Then $\Osc^k (X, x)$ is the span of the union of all $\Osc^k ( A, x )$ for all such $A$. \end{remark}

\begin{definition} \label{DefnInflLocus} Let $\psi \colon X \dashrightarrow \PP^n$ be as above. The \textsl{$k$th inflectional locus} $\Phi^k$ is defined by
\[ \Phi^k \ := \ \left\{ x \in X : \dim \left( \Osc^k ( X, x ) \right) \ < \ d^k \right\} \ = \ \left\{ x \in X : \rank \left( j^k_x \right) \ < \ d^k + 1 \right\} . \]
\end{definition}

\noindent In particular, $\Phi^k$ is a determinantal variety. Hence, if $\Phi^k$ is nonempty, then
\begin{equation} \codim ( \Phi^k , X ) \ \le \ n + 1 - d^k . \label{ExpCodim} \end{equation}
%By for example \cite[Chapter 2]{ACGH},
%\[ (\rank (\Oc \otimes H^0 ( \cL )) - d^k ) ( \rank ( \Im(j^k) - d^k) = (n + 1 - d^k )( d^k + 1 - d^k ) ) = \ n + 1 - d^k \]
%where $\dim (V) = n + 1$. 
Moreover, as there are natural surjections $\cP^{k+1} ( \cL ) \twoheadrightarrow \cP^k ( \cL )$ for $k \ge 0$, we have $\Phi^k \subseteq \Phi^{k+1}$.
%If $y \in \Phi^k$ then $\Image ( P^k|_y^* )$ is too small. Hence $\Image ( P^{k+1}|_y^* )$ is also too small and $y \in \Phi^{k+1}$.

\subsection{Scrolls over curves} \label{scrolls} Let $C$ be a projective smooth curve of genus $g \ge 1$, and $E$ a vector bundle of rank $r$ over $C$. Write $S := \PP E$ and let $\pi \colon S \to C$ be the projection. Let $\cL$ be the relative hyperplane bundle $\cO_{\PP E} (1) \to S$. If $M \to C$ is a line bundle, we write $\cL_M := \cO_{\PP E} (1) \otimes \pi^* M$ to ease notation. By the projection formula, $H^0 ( S, \cL_M ) \cong H^0 ( C, E^* \otimes M )$. Let us describe this identification explicitly in local coordinates.

Suppose $p \in C$ and $x \in S|_p$. Let $U$ be a neighbourhood of $p$ in $C$ over which $E$ and $M$ are trivial. Let $z$ be a uniformiser at $p$ and let $w_1$ be a section of $E \otimes M^{-1}$ near $p$ such that $w_1 (p)$ spans the line $x \in \PP (M^{-1} \otimes E|_p) = S|_p$. Complete $w_1$ to a frame $w_1 , \ldots , w_r$ for $(E \otimes M^{-1})|_U$ and let $\phi_1, \ldots , \phi_r$ be the dual frame for $E^* \otimes M|_U$. Then a section of $E^* \otimes M$ over $U$ is given by an expression
\begin{equation} \sum_{j \ge 0} z^j \cdot \left( \sum_{i=1}^r \alpha_{j, i} \cdot \phi_i \right) \label{SectionOfWd} \end{equation}
where the $\alpha_{j, i}$ are scalars. To view this as a section of $\cL_M$ near $x$, we note that the $\phi_i$ define homogeneous coordinates on the factor $\PP^{r-1}$ of $U \times \PP^{r-1} \cong S|_U$. We restrict to the open subset $\{ \phi_1 \ne 0 \}$ and set $u_i = \phi_i / \phi_1$ for $2 \le i \le r$. In the coordinates $\pi^* z, u_2, \ldots , u_r$ the point $x$ is $( 0, \ldots , 0 )$. Henceforth, we will abuse notation and write $z$ for the local function $\pi^* z$ on $S|_U$. On the set $\{ \phi_1 \ne 0 \}$, the above section (\ref{SectionOfWd}) can be written
\begin{equation} \phi_1 \cdot \sum_{j \ge 0} z^j \cdot \left( \alpha_{j, 1} + \sum_{i = 2}^r \alpha_{j, i} \cdot u_i \right) . \label{SectionOfL} \end{equation}
In the sequel, we will use the expressions (\ref{SectionOfWd}) and (\ref{SectionOfL}) repeatedly.

Furthermore, let us find the expected dimension of $\Phi^k$ in this case. By (\ref{SectionOfL}), clearly $\frac{\partial^2 s}{\partial u_i \partial u_{i'}} = 0$ for all $s \in H^0 ( X, \cL_M )$ and for $2 \le i, i' \le r$. Hence $\Osc^k ( S, x )$ is spanned by differential operators of order at most 1 in the $\frac{\partial}{\partial u_i}$; more precisely, by
\[ \ev_x , \frac{\partial}{\partial z} , \ldots , \frac{\partial^k}{\partial z^k} \quad \hbox{together with} \quad \frac{\partial^\ell}{\partial z^\ell} \frac{\partial}{\partial u_2} , \ldots , \frac{\partial^\ell}{\partial z^\ell} \frac{\partial}{\partial u_r} \quad \hbox{ for } 0 \le \ell \le k-1 . \]
In particular, $d^k \le kr$. If $V \subseteq H^0 ( S, \cL_M )$ is a fixed subspace of dimension $n+1$, set
\begin{equation} k' \ = \ k'_n \ := \ \max\{ k \ge 0 : kr \le n \} . \label{kOne} \end{equation}
Assuming $d^k = kr$ for $0 \le k \le k'$, by (\ref{ExpCodim}) the expected dimension of $\Phi^k$ is 
\begin{equation} \begin{cases} (k + 1)r - n - 1 \quad & \hbox{ if } k = k' ; \\
-1 & \hbox{ if } 0 \le k < k'. \end{cases} \label{ExpDim} \end{equation}

\begin{notation} \label{ntn} Until {\S} \ref{ScrollsIncSyst}, the linear system will always be complete; that is, $V = H^0 ( S, \cL_M )$. However, as $M$ may vary, we write $\Osc^k ( X, x ; \cL_M )$ and $\Phi^k ( \cL_M )$ and $d^k ( \cL_M )$ where necessary. \end{notation}

\section{Osculating spaces via principal parts} \label{SectionOscPpts}
	
Let $\psi \colon S \dashrightarrow | \cL_M |^*$ be as in {\S} \ref{scrolls}. We will describe the osculating spaces $\Osc^k ( S, x) = \Osc^k (S, x; \cL_M )$ using principal parts. As will be indicated below, this is a familiar approach for $k = 0$ and $k = 1$, and is convenient for proofs.

\subsection{Vector bundle-valued principal parts} %(See \cite[Chapter 3]{K1} for more information on the line bundle case, and \cite{Hit1} for higher rank.)
 For any locally free sheaf $F$ over $C$, we have a sequence of $\Oc$-modules
\begin{equation} 0 \ \to \ F \ \to \ \sRat (F) \ \to \ \sPrin (F) \ \to \ 0 \label{RatPrinSheafSeq} \end{equation}
where $\sRat(F)$ is the sheaf of rational sections of $F$, and $\sPrin(F) = \sRat(F)/F$ the sheaf of \textsl{principal parts with values in $F$}. We write $\Rat(F)$ and $\Prin(F)$ respectively for their groups of global sections. As both of these sheaves are flasque, there is an exact sequence
\begin{equation} 0 \ \to \ H^0 ( C, F ) \ \to \ \Rat (F) \ \to \ \Prin (F) \xrightarrow{\partial} H^1 (C, F) \ \to \ 0 . \label{cohomseq} \end{equation}
An element $q \in \Prin ( F )$ can be represented by a collection $( q_p : p \in C )$ where each $q_p$ is a germ of a rational section of $F$ near $p$, and $q_p$ is regular for all but finitely many $p$. We write $\overline{f}$ for the principal part of $f \in \Rat (F)$, and $\partial( q )$ for the class in $H^1 ( C, F )$ of $q \in \Prin (F)$. By exactness, $\partial ( \overline{f} ) = 0$ for all $f \in \Rat (F)$. When we need to specify $F$, we write $\partial_F$ for $\partial$.

\begin{caution} In the literature, the term ``principal part'' is sometimes used as a synonym for ``jet'', which is a different object from a section of $\sPrin (F)$. \end{caution}

\subsection{Motivation} \label{MotivationPptOsc} Let $E \to C$ be a vector bundle and $M \in \Pic^0 (C)$. As $C$ has dimension $1$, by Serre duality and the discussion in {\S} \ref{scrolls}, we have identifications
\begin{equation} H^1 ( C, \Kc M^{-1} \otimes E ) \ \cong \ H^0 ( C, E^* \otimes M )^* \ \cong \ H^0 ( S, \cL_M )^* . %\ = \ H^0 ( S, \cO_{\PP E} (1) \otimes \pi^* M )^* 
\label{SdPf} \end{equation}
Thus $\psi (S) \subseteq |\cL_M|^*$ can be regarded naturally a subvariety of $\PP H^1 ( C, \Kc M^{-1} \otimes E )$. This can be realised concretely using principal parts as follows. Let $p$ be a point of $C$ and $x \in S|_p$. As in {\S} \ref{scrolls}, let $w_1$ be a local section of $M^{-1} \otimes E$ such that $w_1 (p)$ spans the line $x \in \PP (M^{-1} \otimes E)|_p = S|_p$, and let $z$ be a uniformiser on $C$ at $p$. Then the point $\psi(x) = \Osc^0 ( S, x )$ is defined by the cohomology class
\[ \partial \left( \frac{dz \otimes w_1}{z} \right) \ \in \ H^1 ( C, \Kc M^{-1} \otimes E ) . \]
This approach was utilised in \cite{KS} for $E = \Oc = M$, and in \cite{CH2}, \cite{CH5} and elsewhere for bundles of higher rank. Generalising, we now use principal parts to describe $\Osc^k (S, x)$ directly as a subspace of $\PP H^1 (C, \Kc M^{-1} \otimes E )$ for all $k \ge 0$.

\subsection{Osculating spaces via principal parts}

We continue to use the notation of the last subsection. Moreover, as in {\S} \ref{scrolls}, we extend $w_1$ to a frame $w_1 , \ldots , w_r$ for $M^{-1} \otimes E$ near $p$, and let $\phi_1 , \ldots , \phi_r$ be the dual frame for $E^* \otimes M$. 

Now let $\Pi^k_x$ be the $\K$-linear span of the principal parts
\begin{equation} \frac{dz \otimes w_1}{z^{k+1}} \quad \hbox{and} \quad \frac{dz \otimes w_i}{z^j} : \ 1 \le i \le r ; \ 1 \le j \le k . \label{OscClasses} \end{equation}
It is not hard to see that $\Pi^k_x$ is a $\K$-vector subspace of $\Prin ( \Kc M^{-1} \otimes E )$ of dimension $kr + 1$, and that $\Pi^k_x$ is independent of the choice of frame and uniformiser. It depends on the line bundle $M$, but this will always be clear from the context. The coboundary map of (\ref{cohomseq}) restricts to a map $\partial \colon \Pi^k_x \ \to \ H^1 ( C, \Kc M^{-1} \otimes E)$.

\begin{proposition} \label{OscPpts} Via the identification (\ref{SdPf}), we have $\Osc^k ( S, x ) = \PP \left(\partial \left( \Pi^k_x \right) \right)$. \end{proposition}

\begin{proof} By the identification $H^0 ( S, \cL_M ) \isom H^0 ( C, E^* \otimes M )$ described in {\S} \ref{scrolls}, Serre duality defines a perfect pairing
\[ H^0 ( S, \cL_M ) \times H^1 ( C, \Kc M^{-1} \otimes E ) \ \to \ H^1 ( C, \Kc ) \ = \ \K . \]
In view of (\ref{DefnOscIm}) with $V = H^0 ( S, \cL_M)$, and by linear algebra, it suffices to show that under this pairing, $H^0 ( S, \cL_M \otimes \cI_x^{k+1} )$ coincides with $\partial \left( \Pi^k_x \right)^\perp \subseteq H^1 ( C, \Kc M^{-1} \otimes E )^*$.

We will need the commutative diagram of possibly infinite dimensional $\K$-vector spaces
\[ \xymatrix{ H^0 ( C, E^* \otimes M ) \times \Prin ( \Kc M^{-1} \otimes E ) \ar[r]^-{\langle \: , \, \rangle} \ar[d]_{\Iden \times \partial_{\Kc M^{-1} \otimes E}} & \Prin ( \Kc ) \ar[d]_{\partial_{\Kc}} \\
H^0 ( C, E^* \otimes M ) \times H^1 ( C, \Kc M^{-1} \otimes E ) \ar[r]^-{\Sigma ( \: , \, )} & H^1 ( C, \Kc ) }  \]
where $\Sigma$ is the Serre duality pairing, and both $\Sigma$ and $\langle \: , \, \rangle$ are induced by the trace pairing $(E \otimes M) \otimes (E \otimes M)^* \to \Oc$, twisted by the canonical bundle $\Kc$.
%Let us now show that $\Sigma \left( H^0 ( S, \cL_M \otimes \cI_x^{k+1} ) , \partial ( \Pi^k_x ) \right) = 0$. 

Now suppose $s \in H^0 ( S, \cL_M \otimes \cI_x^{k+1} )$. In the local coordinates (\ref{SectionOfL}), the restriction of such an $s$ must be of the form
\[ \phi_1 \cdot z^k \cdot \left( \alpha_{k, 2} u_2 + \cdots + \alpha_{k, r} u_r \right) + \hbox{multiple of } z^{k+1} . \]
Viewing $s$ as a section of $E^* \otimes M \to C$ as in (\ref{SectionOfWd}), near $p$ we obtain the local expression
\[ z^k \cdot \left( \alpha_{k, 2} \phi_2 + \cdots + \alpha_{k, r} \phi_r \right) + \hbox{multiple of } z^{k+1} . \]
Now let $q = \frac{dz \otimes w}{z^\ell}$ be any element of $\Pi^k_x$. Here $1 \le \ell \le k+1$, and $w$ is an $M^{-1}\otimes E$-valued germ near $p$. If $\ell \le k$ then clearly $\langle s, q \rangle$ is everywhere regular, so is trivial in $\Prin ( \Kc )$. If $\ell = k+1$ then by definition of $\Pi^k_x$ we have $w \equiv \lambda w_1 \mod z$ for some $\lambda \in \K$. Since $\phi_i (w_1) = 0$ for $2 \le i \le r$, again $\langle s , q \rangle$ is regular. Thus in either case the cohomology class
\[ \Sigma \left( s , \partial_{\Kc M^{-1} \otimes E} (q) \right) \ = \ \partial_\Kc \langle s, q \rangle \ \in \ H^1 (C, \Kc) \]
is zero for all $q \in \Pi^k_x$. Therefore, $H^0 ( S, \cL_M \otimes \cI_x^{k+1} ) \subseteq \partial ( \Pi^k_x )^\perp$.

%Let us now show that $H^0 ( S, \cL_M \otimes \cI_x^{k+1} )$ is the largest space annihilated by $\delta(\Pi^k_x)$.
Conversely, suppose $s \not\in H^0 ( S, \cL_M \otimes \cI_x^{k+1} )$; equivalently, $j^k_x (s) \ne 0$. Viewing $s$ as a section of $E^* \otimes M \to C$ as above, the local expression near $p$ is
\[ z^\ell f_\ell + \hbox{multiple of } z^{\ell + 1} \]
where $f_\ell$ is a nonzero linear combination of $\phi_1 , \ldots , \phi_r$. Since $j^k_x (s)$ is nonzero, $\ell \le k$, and if $\ell = k$ then $f_\ell ( w_1 ) \ne 0$. If $\ell = k$ then set $w = w_1$. Otherwise, let $w$ be any linear combination of $w_1 , \ldots , w_r$ such that $f_\ell (w) \ne 0$. Then the principal part $\frac{dz \otimes w}{z^{\ell + 1}}$ belongs to $\Pi^k_x$. We compute
\[ \left\langle s, \frac{dz \otimes w}{z^{\ell + 1}} \right\rangle \ = \ f_{\ell}(w) \cdot \frac{dz}{z} \ \in \ \Prin ( \Kc ) . \]
As $f_{\ell} (w) \ne 0$, this has a pole of order exactly $1$ at $p$, so defines a nonzero class in $H^1 ( C, \Kc )$ since $g \ge 1$. Hence $s \not\in \partial ( \Pi^k_x )^\perp$.

This establishes equality $H^0 ( S, \cL_M \otimes \cI_x^{k+1} ) = \partial ( \Pi^k_x )^\perp$, which completes the proof. \end{proof}

\begin{remark} The notation could be simplified by working with the spaces $H^1 ( C, E \otimes M )$ instead of $H^1 ( C, \Kc M^{-1} \otimes E )$. The reason we have not done so is that we wish to obtain scrolls in $\PP H^0 ( C, E^* \otimes M )^* = \PP H^0 ( S, \cO_{\PP E} (1) \otimes \pi^* M )^*$, in order to maintain the connection with \cite{LMP} and other works on this topic. \end{remark}

\subsection{Osculating spaces and elementary transformations}

This logically independent subsection is included to show the connection between the Proposition \ref{OscPpts} and some existing descriptions of the embedded tangent spaces to models of other fibrations over $C$.

\begin{corollary} \label{OscET} For $x \in S|_p$, let $0 \to E \to \tW \to \K_p \to 0$ be the elementary transformation such that $\Ker ( E|_p \to \tW|_p )$ is the line $x$. Then $\Osc^k ( S, x )$ coincides with the projectivised image of the coboundary map in
\[ H^0 ( C, \Kc M^{-1} \otimes \tW (kp) ) \ \to \ H^0 \left( C, \frac{\Kc M^{-1} \otimes \tW (kp)}{\Kc M^{-1} \otimes E} \right) \ \xrightarrow{\partial} \ H^1 ( C, \Kc M^{-1} \otimes E ) . \] \end{corollary}

\begin{proof} The elementary transformation $\tW$ can be realised as the subsheaf of $\sRat (E)$ of sections regular away from $p$ and with at most simple poles at $p$ in the direction of $w_1$. Thus $M^{-1} \otimes \tW_p$ is spanned by $\frac{w_1}{z}, w_2 , \ldots , w_r$, and $M^{-1} \otimes \tW ( kp )_p$ is spanned by
\[ \frac{w_1}{z^{k+1}}, \frac{w_2}{z^k}, \ldots , \frac{w_r}{z^k} . \]
It follows from the description (\ref{OscClasses}) that $H^0 \left( C, \frac{\Kc M^{-1} \otimes \tW (kp)}{\Kc M^{-1} \otimes E} \right)$ is exactly $\Pi^k_x$. Hence the corollary follows from Proposition \ref{OscPpts}. \end{proof}

\begin{remark} \label{ConnectionToLiterature} Corollary \ref{OscET} is a result of the same type as \cite[Lemma 5.3]{CH1}, \cite[Lemma 3.5]{CH2} and \cite[Lemma 3.2]{CH5} %(see also \cite[Proposition 4.2]{Hit17}), 
 which describe the embedded tangent spaces to various quadric and Grassmannian fibrations over $C$ in terms of elementary transformations. (The applications to Segre invariants studied in these papers are however of a different nature to that in {\S} \ref{SectionSegre}.) \end{remark}

\section{Quot schemes and inflectional loci} \label{SectionQuotInfl}

\subsection{Motivation} We will now use Quot schemes to describe inflectional loci of scrolls. To motivate the use of Quot schemes in this context, we note that in the situation of Proposition \ref{OscPpts}, the osculating space $\Osc^k ( S, x )$ fails to be of dimension $kr$ if and only if
\[ \partial \left( \frac{dz \otimes w}{z^\ell} \right) \ = \ 0 \ \in \ H^1 ( C, \Kc M^{-1} \otimes E ) \]
for certain $\ell$ and $w$. By exactness of (\ref{cohomseq}), the principal part $\frac{dz \otimes w}{z^\ell}$ arises from a global section of $H^0 ( C, \Kc M^{-1} \otimes E ( \ell p ))$, giving a sheaf injection $\Oc ( - \ell p ) \to \Kc M^{-1} \otimes E$. Thus the existence of inflection points implies the existence of invertible subsheaves of the form $\Oc ( - \ell p ) \subset \Kc M^{-1} \otimes E$; and the latter define points in a Quot scheme of $\Kc M^{-1} \otimes E$. In the following we will make this correspondence precise.

\subsection{Quot schemes of invertible subsheaves} Let $E$ be a vector bundle of rank $r$ and degree $d$ over $C$, and let $M$ be a line bundle of degree zero. For any integer $\ell$, invertible subsheaves of degree $\ell$ of $\Kc M^{-1} \otimes E$ are in canonical bijection with coherent quotients of rank $r-1$ and degree
\[ \deg ( \Kc M^{-1} \otimes E ) - \ell \ = \ d + 2r(g-1) - \ell . \]
Such quotients are parametrised by the Quot scheme $\Quot^{r-1, \, d + 2r(g-1) - \ell} ( \Kc M^{-1} \otimes E )$. To ease notation, and to emphasise that we are primarily interested in subsheaves, we denote this by $Q_{\ell} ( \Kc M^{-1} \otimes E )$. For points of this scheme we write $[ \sigma \colon L \to \Kc M^{-1} \otimes E ]$, where $\sigma$ is a sheaf injection and $L \in \Pic^\ell (C)$. If there is no ambiguity, we may simply write $\sigma$.

%To ease notation, set $\ell := -(k+1)$.
\subsection{A parameter space} For any integer $\ell$, let $a_\ell \colon C \to \Pic^\ell (C)$ be the map $p \mapsto \Oc ( \ell p )$, and
\[ b \colon Q_\ell ( \Kc M^{-1} \otimes E ) \ \to \ \Pic^{\ell} (C) \]
the forgetful map $[ \sigma \colon L \to \Kc M^{-1} \otimes E ] \ \mapsto \ L$. We have the fibre product
\begin{equation} \label{defnRkM} \xymatrix{ R^k_M \ar[rr] \ar[d] & \quad & C \ar[d]^{a_{-(k+1)}} \\
 Q_{-(k+1)}(\Kc M^{-1} \otimes E) \ar[rr]^-{b} & & \Pic^{-(k+1)} ( C ) . } \end{equation}
Set-theoretically, $R^k_M$ is the set of pairs of the form
\[ \left( p, \left[ \sigma \colon \Oc ( -(k+1)p ) \to \Kc M^{-1} \otimes E \right] \right) . \]
%By definition of the Quot scheme and since $\rank ( \Oc ( -(k+1) p) ) = 1$, the fibre of $R^k_M$ over $p \in C$ is $\PP H^0 ( C, \Hom ( \Oc ( -(k+1) p ) , \Kc M^{-1} \otimes E ))$. On the other hand, the fibre over $[ \sigma \colon \Oc ( -(k+1) p ) \to \Kc M^{-1} \otimes E ]$ is
%\[ \{ q \in C : \Oc ( (k+1) q ) \cong \Oc ( (k+1) p ) \} . \]
We define a map $e_k \colon R^k_M \dashrightarrow \PP E =: S$ by
\[ \left( p, \left[ \sigma \colon \Oc ( -(k+1) p ) \to \Kc M^{-1} \otimes E \right] \right) \ \mapsto \ \Image ( \sigma|_p ) \ \in \ S|_p . \]
This is defined at $(p, \sigma)$ if and only if $\sigma$ is a vector bundle injection at $p$. (Note that $e_k$ also depends on $M$, but $M$ will always be clear from the context.)

\begin{proposition} \label{ParamInfl} Let $E$ be a bundle of rank $r$ and degree $d$. Let $x$ be a point of a fibre $S|_p$. For $k \ge 0$ and $M \in \Pic^0 (C)$, the following holds. \label{ImageOfE}
\begin{enumerate}
\item[(a)] If $x \in e_k ( R^k_M )$ then $\dim ( \Osc^k ( S, x ; \cL_M ) ) < kr$.
\item[(b)] If $\dim ( \Osc^k ( S, x ; \cL_M ) ) < kr$, then $x \in \Image ( e_k )$ or $\Image ( e_\ell ) \cap S|_p$ is nonempty for some $\ell < k$. %In particular, if $k = 0$ then $x \in \Image ( e_0 )$.
\end{enumerate} \end{proposition}

\begin{proof} (a) Suppose $x = e_k ( p, \sigma )$ for some $( p, \sigma ) \in R^k_M$. Then
\[ \sigma \ \in \ H^0 ( C, \Hom ( \Oc( -(k+1) p ) , \Kc M^{-1} \otimes E )) \ = \ H^0 ( C, \Kc M^{-1} \otimes E ( (k+1)p ) ) \]
is a sheaf injection which is saturated at $p$. We view $\sigma$ as a rational section of $\Kc M^{-1} \otimes E$ with a pole of order at most $k+1$ at $p$. Since $\sigma$ is saturated, in fact the pole has order exactly $k+1$. Hence the principal part $\overline{\sigma}$ is of the form
\[ \frac{dz \otimes w}{z^{k+1}} \]
for some local section $w$ of $M^{-1} \otimes E$ which is nonzero at $p$, and such that
\[ x \ = \ \Image( \sigma|_p ) \ \in \ \PP (\Kc M^{-1} \otimes E)|_p \ = \ S|_p , \]
so $x$ is the line spanned by $w (p)$. Thus $\overline{\sigma}$ is a nonzero element of $\Pi^k_x$ (cf.\ \ref{OscClasses}). The exactness of (\ref{cohomseq}) implies that $\partial (\overline{\sigma}) = 0$ in $H^1 ( C, \Kc M^{-1} \otimes E)$. By Proposition \ref{OscPpts}, the osculating space $\Osc^k ( S, x ; \cL_M )$ has dimension strictly less than $\dim ( \PP ( \Pi^k_x ) ) = kr$. %Hence $e^k ( p, \sigma )$ belongs to $\Phi^k ( \cL_M )$.

(b) Suppose $\dim ( \Osc^k ( S, x ; \cL_M ) ) < kr$. By Proposition \ref{OscPpts}, there exists a nonzero
\[ q \ \in \ \Ker \left( \partial \colon \Pi^k_x \ \to \ H^1 ( C, \Kc M^{-1} \otimes E ) \right) . \]
Such a $q$ has the form $\frac{dz \otimes w}{z^{\ell + 1}}$, where $0 \le \ell \le k$ and $w$ is a local section of $M^{-1} \otimes E$ which is nonzero at $p$. By exactness of (\ref{cohomseq}), there exists
\[ \sigma \ \in \ H^0 \left( C, \Kc M^{-1} \otimes E ((\ell+1)p) \right) \ = \ H^0 ( C, \Hom ( \Oc ( -(\ell+1) p ) , \Kc M^{-1} \otimes E )) \]
with $\overline{\sigma} = q$, which is a vector bundle injection at $p$. Let $x' \in S|_p$ be the line spanned by $w(p)$. Then
\[ x' \ = \ \Image ( \sigma (p) ) \ = \ e_{\ell} \left( p, [ \sigma \colon \Oc ( - (\ell+1)p ) \to \Kc M^{-1} \otimes E ] \right) . \]
If $\ell = k$ then by definition of $\Pi^k_x$ the section $w$ spans the line $x$ at $p$, and $x' = x = e_k ( p, \sigma )$. If $\ell < k$ then $x' \in S|_p \cap \Image ( e_\ell )$, as desired. \end{proof}

%In particular, since $q$ is nonzero in $\Prin ( \Kc M^{-1} \otimes E )$, if $k = 0$ then necessarily $m = 1$. %From this and part (a) it follows that $\Image (e^0)$ is exactly $\Phi^0 ( \cL_M )$, the base locus of $|\cL_M|$.

\begin{comment} %Other approach:
\begin{enumerate}
\item[(a)] If $x \in e_k ( R^k_M )$ then $\partial \colon \Pi^k_x  \to H^1 ( C, \Kc M^{-1} \otimes E )$ is not injective.
\item[(b)] If $\partial \colon \Pi^k_x \to H^1 ( C, \Kc M^{-1} \otimes E )$ is not injective, then either $x \in \Image ( e_k )$ or $\partial \colon \Pi^{k-1}_{x'} \to H^1 ( C, \Kc M^{-1} \otimes E )$ is not injective for some $x' \in S|_p$. In particular, if $k = 0$ then $x \in \Image ( e_0 )$.
\end{enumerate}

\begin{proposition} Let $x$ be a point of $S$, and write $p = \pi(x)$.
\begin{enumerate}
\item[(a)] Suppose $x \in \Image ( e_k )$. Then $\dim ( \Osc^k ( S, x ) ) < kr$.
\item[(b)] Suppose $\dim ( \Osc^k ( S, x ) ) < kr$. Then $x \in \Image( e_k )$ or $\Image ( e_{\ell} ) \cap S|_p$ is nonempty for some $\ell < k$.
\end{enumerate}
\end{proposition}
\end{comment}

\begin{caution} We may have $R^k_M \ne \emptyset$ even if $\dim ( \Osc^k ( S, x ; \cL_M ) ) = kr$ for all $x$. If
\begin{equation} H^0 ( C, \Kc M^{-1} \otimes E ) \ \cong \ H^0 ( C, \Kc M^{-1} \otimes E ( (k+1)p )) \hbox{ for all } p \in C , \label{StrangeInflection} \end{equation}
then $\partial \colon \Pi^k_x \to H^1 ( C, \Kc M^{-1} \otimes E )$ is injective for all $x \in S$, and $\dim ( \Osc^k ( S, x ; \cL_M ) ) = kr$ by Proposition \ref{OscPpts}. But in this situation, if $s$ is a nonzero section of $\Kc M^{-1} \otimes E$, then $R^k_M$ contains the point
\[ \left( p, \left[ \sigma \colon \Oc ( - (k+1)p ) \to \Oc \xrightarrow{s} \Kc M^{-1} \otimes E \right] \right) . \]
This does not contradict Proposition \ref{ParamInfl}, because $(p, \sigma)$ is a point of indeterminacy for $e_k$; by (\ref{StrangeInflection}), the map $\sigma$ is not saturated at $p$. The distinction between nonemptiness of $R^k_M$ and that of the image $e_k ( R^k_M )$ is significant. Clearly this situation can arise only if $\Kc M^{-1} \otimes E$ is \emph{special}; that is, both $h^0 ( C, \Kc M^{-1} \otimes E )$ and $h^1 ( C, \Kc M^{-1} \otimes E )$ are nonzero.

An example of such a $E$ for $k = 0$ can be given as follows. Suppose $C$ is general of genus $g \ge 4$, and set $M = \Oc$. Let $N \in \Pic^{-(g+1)}(C)$ be general, and suppose $L \in \Pic^{-(g+2)} (C)$ satisfies $h^0 ( C, \Kc L ) = 1$ and $|L^{-1}|$ base point free. Then one can check that any extension $0 \to L \to E \to N \to 0$ satisfies $h^0 ( C, \Kc \otimes E ) = h^0 ( C, \Kc \otimes E (p)) \ne 0$ for \emph{all} $p \in C$.

\begin{comment}
Details: $\Kc N \in \Pic^{g-3} (C)$. Since $C$ is general (Petri), $\dim ( W^0_{g-2} ) = g - 2$, so it does not meet the curve
\[ \{ \Kc N (p) : p \in C \} \ \subset \ \Pic^{g-2} (C) \]
for general enough $N$. Hence $h^0 ( \Kc N (p) ) = h^0 ( \Kc N ) = 0$ for all $p \in C$.

Next, $\deg ( \Kc L ) = 2g-2-g-2 = g-4$. Suppose $\Kc L \in W^0_{g - 4}$. Then $h^1 ( \Kc L ) = h^0 ( L^{-1} ) = 4$. We want $h^0 ( \Kc L (p) ) = 1$; that is, $h^1 ( \Kc L (p) ) = h^0 ( L^{-1} (-p) ) = 3$, for all $p \in C$. If $L^{-1} \in W^3_{g+2}$ is base point free then this is assured.

We need $g \ge 4$ since $C$ is general, to ensure that $W^0_{g-4}$ is nonempty. If $g = 4$ we can take $L = T_C$, as $\Kc$ is base point free. Here $g + 2 = 2g - 2$ as $4 + 2 = 2 \cdot 4 - 2$.

Now consider the diagram
\[ \xymatrix{ & 0 \ar[d] & 0 \ar[d] & 0 \ar[d] & \\
0 \ar[r] & H^0 ( \Kc L ) \ar[r] \ar[d] & H^0 ( \Kc \otimes E ) \ar[r] \ar[d] & H^0 ( \Kc N ) \ar[r] \ar[d] & \cdots \\
0 \ar[r] & H^0 ( \Kc L (p) ) \ar[r] \ar[d] & H^0 ( \Kc \otimes E (p) ) \ar[r] \ar[d] & H^0 ( \Kc N (p) ) \ar[r] \ar[d] & \cdots \\
 & \vdots & \vdots & \vdots & } \]
As $H^0 ( \Kc N ) = H^0 ( \Kc N (p) ) = 0$ and $h^0 ( \Kc L ) = h^0 ( \Kc L (p) ) = 1$ for all $p \in C$, by exactness we have 
\[ h^0 ( \Kc \otimes E ) \ = \ h^0 ( \Kc L ) \ = \ h^0 ( \Kc L (p) ) \ = \ h^0 ( \Kc \otimes E (p) ) \ = \ 1 \]
for all $p \in C$.
\end{comment}
%\fix{A more compact (but less useful) statement of Proposition \ref{ParamInfl} is that $R^k_M$ is nonempty if and only if $\codim ( \Osc^k ( S, x ; \cL_M ) ) > \chi( E ) - kr$.}
\end{caution}

\section{Segre invariant and dimensions of osculating spaces} \label{SectionSegre}

In this section, we will give the first application of Proposition \ref{ParamInfl}, linking the degrees of invertible subsheaves of a bundle $E \to C$ to the dimensions of osculating spaces to $S = \PP E$. We use the language of Segre invariants, as defined in (\ref{defnSegreInv}). This will lead to a geometric characterisation of the semistability and cohomological stability properties. Let us first state an important result on Segre invariants.

Taking direct sums of line bundles, it is easy to find a bundle $E$ with $s_n (E)$ arbitrarily small. Conversely, however, Hirschowitz \cite{Hir2} determined the following upper bound on $s_n (E)$, valid for all $E$ (see also \cite{CH1}).
\begin{theorem} \label{HirschBound} Let $E$ be a bundle of rank $r \ge 2$ and degree $d$. For $1 \le n \le r - 1$, we have $s_n (E) \le n (r-n)(g-1) + \delta$, where $\delta \in \{ 0 , \ldots , r-1 \}$ satisfies $n(r-n)(g-1) + \delta \equiv nd \mod r$. \end{theorem}

\noindent We now give a characterisation of $s_1 (E)$ using osculating spaces.

\begin{theorem} \label{MainA} Suppose $k \ge 0$, and let $E$ be a bundle of rank $r$ and degree $d$. Then the following are equivalent.
\begin{enumerate}
\item[(1)] $s_1 ( E ) > d + r ( 2g - 1 + k )$.
\item[(2)] For all $M \in \Pic^0 ( C )$ and all $x \in S$, the osculating space $\Osc^k ( S, x ; \cL_M )$ has dimension $kr$.
\item[(3)] For all $M \in \Pic^0 ( C )$ we have $d^k (\cL_M) = kr$ and $\Phi^k ( \cL_M )$ is empty.
\end{enumerate}
\end{theorem}
\begin{remark} By Theorem \ref{HirschBound}, we have $s_1 ( E ) \le (r-1)g$ in any case. Then one computes that (1) can obtain only if
\[ d \ < \ -(r+1)(g - 1) - (kr + 1) \ \le \ -(r+1)(g-1) - 1 , \]
the latter since $k \ge 0$. We will use this line of argument again in Corollary \ref{SpecialCases} (b). \end{remark}
\begin{comment}
Details: (1) can obtain only if
\begin{align*} d + r(2g - 1 + k) \ & < \ (r-1)(g-1) + \delta \ = \ r(g - 1) - (g-1) + \delta\\
d \ & < \ r ( g - 1 - 2g + 1 - k ) - (g-1) + \delta \\
d \ & < \ r ( - g + 1 - (k + 1) ) - (g-1) + \delta \\
d \ & < \ -(g-1)(r+1) - (k+1)r + \delta \\
d \ & < \ -(g-1)(r+1) - (k+1)r + \delta \\
d \ & < \ -(g-1)(r+1) - (k+1)r + r - 1 \\
d \ & < \ -(g-1)(r+1) - (kr + 1) \\
\end{align*}
\end{comment}

\begin{proof} The equivalence of (2) and (3) follows from the definitions. We will prove the equivalence of (1) and (2). Suppose $s_1 ( E ) > d + r ( 2g - 1 + k )$. Recall that $s_1 ( \Kc M^{-1} \otimes E ) = s_1 ( E )$ for any line bundle $M$. Thus for any $M \in \Pic^0 (C)$, each invertible subsheaf $L$ of $\Kc M^{-1} \otimes E$ satisfies
\[ \deg ( \Kc M^{-1} \otimes E ) - r \cdot \deg ( L ) \ > \ d + r ( 2g - 1 + k ) , \]
that is, $\deg (L) < - (k + 1)$. Thus $Q_{-(\ell+1)} ( \Kc M^{-1} \otimes E )$ is empty for $\ell \le k$. In particular, the fibre product $R^\ell_M$ defined in (\ref{defnRkM}) is empty for $0 \le \ell \le k$. By Proposition \ref{ParamInfl} (b), we have $\dim ( \Osc^k ( S, x ; \cL_M ) = kr$ for all $M$ and all $x$, establishing (2).

Conversely, suppose $s_1 ( E ) \le d + r ( 2g - 1 + k )$. Then there is a sheaf injection $\sigma' \colon L \to \Kc \otimes E$ where $L$ is invertible of degree $-(k+1)$. Let $p \in C$ be any point where $\sigma'|_p$ is a vector bundle injection. Set $M = L ( (k+1)p )$, a line bundle of degree zero. Then
\[ \sigma \ := \ \sigma' \otimes \Iden_{M^{-1}} \colon \Oc ( - (k+1)p ) \ \to \ \Kc M^{-1} \otimes E \]
is a vector bundle injection at $p$. This determines a point $(p, \sigma) \in R^k_M$ at which $e_k$ is defined. Writing $x := e_k ( p, \sigma )$, by Proposition \ref{ParamInfl} (a) we have $\dim ( \Osc^k ( S, x ; \cL_M ) ) < kr$. Thus (2) implies (1). \end{proof}

Let us now discuss some special cases.

\begin{corollary} \label{SpecialCases} Suppose $d \le r(1-2g)$. Let $E$ be a vector bundle of rank $r$ and degree $d$, and $S = \PP E$.
\begin{enumerate}
\item[(a)] Suppose $E$ is stable (or more generally, $s_1 (E) > 0$). Then for $k \le \mu ( E^* ) - (2g-1)$ and for all $M \in \Pic^0 (C)$, we have $\dim ( \Osc^k ( S, x ; \cL_M )) = kr$ for all $x \in S$.
\item[(b)] Let $E$ be any bundle of rank $r$ and degree $d$. If
\[ k \ge \mu (E^*) - \frac{(r+1)g - 1}{r} \]
then $\dim ( \Osc^k ( S, x ; \cL_M ) ) < kr$ for some $M \in \Pic^0 ( C )$.
\item[(c)] If $s_1 (E)$ is the generic value (for example, if $E$ is a general stable bundle), then the converse to (b) also holds.
\end{enumerate} \end{corollary}
\begin{proof} (a) Suppose $k \le \mu( E^* ) - (2g-1)$. (The condition on $d$ is to ensure that this can obtain for some $k \ge 0$.) Then $d + r(2g-1 + k) \le 0$.	
\begin{comment}
Details: 
\[ k \ \le \ \mu( E^* ) - (2g-1) \quad \Longleftrightarrow \]
\[ kr \ \le \ -d - r(2g-1) \quad \Longleftrightarrow \]
\[ 0 \ \le \ -d - r(2g - 1 + k ) \quad \Longleftrightarrow \]
\[ 0 \ \ge \ d + r(2g - 1 + k ) \]
\end{comment}
Thus if $s_1 (E) > 0$ then by Theorem \ref{MainA}, for all $M \in \Pic^0 (C)$ we have $\dim ( \Osc^k ( S, x ; \cL_M ) ) = kr$ for all $x \in S$.

(b) As $k$ is an integer, the hypothesis implies that
\begin{equation} kr \ \ge \ - d - (r+1)g + 1 + \delta \label{PartbIneq} \end{equation}
where $0 \le \delta \le r-1$ is such that $- d - (r+1)g + 1 + \delta \equiv 0 \mod r$. One computes that inequality (\ref{PartbIneq}) is equivalent to
\[ (r-1)(g-1) + \delta \ \le \ d + r (2g-1 + k) \]
and that $\delta$ satisfies $(r-1)(g-1) + \delta \equiv d \mod r$.
\begin{comment}
\[ kr \ \ge \ - d - (r+1)g + 1 + \delta \]
\[ d + r(2g-1) + kr \ \ge \ r(2g-1) - (r+1)g + 1 + \delta \]
\[ d + r(2g-1+ k) \ \ge \ 2rg - r - rg - g + 1 + \delta \]
\[ d + r(2g-1+ k) \ \ge \ rg - r - g + 1 + \delta \]
\[ d + r(2g-1+ k) \ \ge \ r(g - 1) - (g - 1) + \delta \]
\[ d + r(2g-1+ k) \ \ge \ (r-1)(g - 1) + \delta \]

Also,
\[ - d - (r+1)g + 1 + \delta \equiv 0 \ \mod \ r \]
\[ - (r+1)(g-1) - r - 1 + 1 + \delta \equiv d \ \mod \ r \]
\[ - (r+1)(g-1) + \delta \equiv d \ \mod \ r \]
\[ 2r(g-1) - (r+1)(g-1) + \delta \equiv d \ \mod \ r \]
\[ (r-1)(g-1) + \delta \equiv d \ \mod \ r \]
\end{comment}
Thus $s_1 (E) \le (r-1)(g-1) + \delta$ by Theorem \ref{HirschBound}. It follows that $s_1 (E) \le d + r(2g-1 + k)$. By Theorem \ref{MainA}, for some $M \in \Pic^0 (C)$ and some $x \in S$ we have $\dim ( \Osc^k ( S, x ; \cL_M ) ) < kr$.

(c) Suppose $s_1 (E)$ is the generic value $(r-1)(g-1) + \delta$ stated in Theorem \ref{HirschBound}. Suppose $\dim ( \Osc^k ( S, x ; \cL_M ) ) < kr$ for some $x \in S$ and $M \in \Pic^0 ( C )$. By Theorem \ref{MainA}, we have
\[ d + r (2g-1+k) \ \ge \ (r-1)(g-1) + \delta . \]
Computing, we obtain $k \ge \mu (E^*) - \frac{1}{r}\left( (r+1)g - 1 \right)$ as desired.
\begin{comment}
\[ d + r (2g-1+k) \ \ge \ (r-1)(g-1) + \delta \]
\[ kr \ \ge \ -d - r (2g-1) + (r-1)(g-1) + \delta \]
\[ kr \ \ge \ -d - 2rg + r + rg - r - g + 1 + \delta \]
\[ kr \ \ge \ -d - rg - g + 1 + \delta \]
\[ kr \ \ge \ -d - (r+1)g + 1 \]
\[ k \ \ge \ \mu (E^*) - \frac{(r+1)g - 1}{r} \]
\end{comment}
\end{proof}

\begin{remark} \label{kZeroOrOne} Let us examine the situation for $k = 0$ and $k = 1$. By definition, $\Phi^0 ( \cL_M )$ is empty if and only if $| \cL_M |$ is base point free; that is, $E^* \otimes M$ is globally generated. Thus from Corollary \ref{SpecialCases} (a) we recover the fact (see \cite[Proposition 2 (ii)]{IT}) that if $E$ is stable of slope $\le -(2g-1)$, then $E^* \otimes M$ is generated for all $M$ of degree zero, so all $S \to |\cL_M|^*$ are base point free. In fact Theorem \ref{MainA} gives a more precise statement: If $\mu (E) \le 1-2g$ then $E^* \otimes M$ is generated for all $M$ of degree zero \emph{if and only if} $s_1 ( E) > 0$. Similarly, if $\mu ( E ) \le -2g$ then the differential of $S \to |\cL_M|^*$ is everywhere injective for all $M$ if and only if $s_1 ( E ) > 0$.
\begin{comment}
Details: Note that in \cite{IT} the notation is $\PP E = \PP ( E^* )$, so their $E$ is our $E^*$. Suppose $d_1 ( E^* ) \ge 2g$. Then $\deg (L) < -(2g-1)$ for all line subbundles $L \subset E$, so $s_1 ( E ) > d - r ( 2g-1 )$. By (one implication of) Theorem \ref{MainA}, the locus $\Phi^0 ( \cL_M )$ is empty for all $M \in \Pic^0 (C)$ [and not just $M = \Oc$]. But $E^*$ is globally generated if and only if $\Phi^0 = \emptyset$.
\end{comment}
\end{remark}

\begin{remark} Theorem \ref{MainA} does not hold for incomplete linear systems. For any $M$, by projecting $|\cL_M|^*$ from a point of $\Osc^k ( S, x )$, one obtains a projective model of $S$ with an inflection point at $x$, regardless of $s_1 (E)$. \end{remark}

\subsection{A geometric characterisation of semistability} \label{semist}

We will now use Theorem \ref{MainA} to give a characterisation of the semistability property for vector bundles of slope less than $1 - 2g$ over $C$. If $E \to C$ is such a vector bundle, for $1 \le n \le r-1$ we write
\[ \pi_n \colon \PP ( \wedge^n E ) \ =: \ S_n \ \to \ C , \]
a projective bundle of dimension $\binom{r}{n}$. If $M$ is a line bundle over $C$, we write $\cL_{n, M} := \cO_{\PP ( \wedge^n E)} ( 1 ) \otimes \pi_n^* M$, and consider the map $S_n \dashrightarrow | \cL_{n, M} |^*$. %Also, set
%\[ \knmax \ := \ \max \left\{ k \ge 0 : k < n \cdot \mu ( E^*) - (2g-1) \right\} . \]
%By hypothesis, we have
%\[ \mu ( \wedge^n E^* ) \ = \ n \cdot \mu ( E^* ) \ \ge \ \mu ( E^* ) \ > \ 2g - 1 , \]
%so $\knmax \ge 0$.

\begin{theorem} \label{MainB} Let $E$ be a bundle of rank $r$ and slope $\mu < 1 - 2g$ over $C$. Then the following are equivalent.
\begin{enumerate}
\item[(1)] $E$ is semistable.
\item[(2)] For $1 \le n \le r-1$ and for $0 \le k < n \cdot \mu ( E^*) - (2g-1)$, the osculating space $\Osc^k ( S_n , y ; \cL_{n, M} )$ is of the expected dimension $k \cdot \binom{r}{n}$ for all $y \in S_n$ and all $M \in \Pic^0 ( C )$.
\end{enumerate} \end{theorem}

\begin{proof} Suppose $E$ is semistable. As $\K$ has characteristic zero, by \cite[Theorem 10.2.1]{LeP}, for $1 \le n \le r$ the exterior product $\wedge^n E$ is semistable of slope $n \cdot \mu (E)$. %https://books.google.no/books?id=okHfUv4l4vgC&printsec=frontcover#v=onepage&q=tensor&f=false
 In particular, $s_1 ( \wedge^n E ) \ge 0$. By Theorem \ref{MainA}, for all $M \in \Pic^0 (C)$ and all $y \in S_n$, the osculating space $\Osc^k ( S_n , y ; \cL_{n, M}) )$ has dimension $k \cdot \binom{r}{n}$ for all $k \ge 0$ such that
\[ \deg ( \wedge^n E ) + \rank ( \wedge^n E ) ( 2g-1 + k ) \ < \ 0 , \]
%\[ \mu ( \wedge^n E ) + ( 2g-1 + k ) \ < \ 0 , \]
%\[ n \cdot \mu ( E ) + (2g-1) + k \ < \ 0 , \]
%\[ k \ < \ n \cdot \mu ( E^* ) - ( 2g-1 ) \]
%that is, $0 \le k \le \knmax$.
that is, $k < n \cdot \mu ( E^*) - (2g-1)$.

Conversely, assume (2) holds. %Now suppose that for $1 \le n \le r-1$ the osculating spaces $\Osc^k ( \PP ( \wedge^n E ) , x ; \cL_{n, M} )$ have dimension $k \cdot \binom{r}{n}$ for all $x$ and $M$ for $0 \le k_n < n \cdot \mu ( E^* ) - (2g-1)$; that is, for $0 \le k_n \le \knmax$. Then 
 By Theorem \ref{MainA}, for $1 \le n \le r-1$ we have
\[ s_1 ( \wedge^n E ) \ > \ \deg ( \wedge^n E ) + \rank ( \wedge^n E ) \cdot ( 2g - 1 + k_n ) , \]
where $k_n = \max \left\{ k \ge 0 : k < n \cdot \mu ( E^*) - (2g-1) \right\}$. By the definition (\ref{defnSegreInv}) of $s_1 ( E )$, the left and right hand sides of the last inequality are congruent modulo $\rank ( \wedge^n E )$. 
%$s_1 ( \wedge^n E ) = d - \rank ( \wedge^n E ) \cdot \deg (L)$ for some line bundle $L$.
 Therefore,
\[ s_1 ( \wedge^n E ) \ \ge \ \deg ( \wedge^n E ) + \rank(\wedge^n E) \cdot ( 2g - 1 + k_n + 1 ) \]
%\[ \ge \ \deg ( \wedge^n E ) + \rank(\wedge^n E) \cdot ( 2g - 1 + n \cdot \mu ( E^*) - (2g-1) ) \ = \ 0 \]
By definition of $k_n$, 
%$\knmax + 1 \ge n \cdot \mu ( E^*) - (2g-1)$.
 this becomes $s_1 ( \wedge^n E ) \ge 0$.
%\[ \deg ( \wedge^n E ) + \rank(\wedge^n E) \cdot ( 2g - 1 + \knmax + 1 ) \ge \deg ( \wedge^n E ) + \rank(\wedge^n E) \cdot ( 2g - 1 + n \cdot \mu ( E^*) - (2g-1) ) \]
%\[ = \deg ( \wedge^n E ) + \rank(\wedge^n E) \cdot ( n \cdot \mu ( E^*) ) \]
%\[ = \deg ( \wedge^n E ) + \rank(\wedge^n E) \cdot \mu ( \wedge^n E^*) \]
%\[ = \deg ( \wedge^n E ) + \deg ( \wedge^n E^*) \ = \ 0 \]

Now suppose $F$ is a rank $n$ subbundle of $E$. Then $\wedge^n F$ is a line subbundle of $\wedge^n E$. Since we have shown that $s_1 ( \wedge^n E ) \ge 0$, it follows that
\[ \mu ( \wedge^n F ) \ = \ \deg (F) \ \le \ \mu ( \wedge^n E ) \ = \ n \cdot \mu (E) , \]
whence $\mu ( F ) \le \mu (E)$. Hence $E$ is semistable. \end{proof}

%\begin{remark} The hypothesis $\mu (E^*) \ge (2g-1)$ is to ensure that for some $k \ge 0$ the osculating spaces can be of the expected dimension. This may happen for other values of $d$ for special $E$, but we want a uniform statement, valid for all $E$. \end{remark}

\subsection{Cohomological stability} \label{CohomSt} The implication $(1) \Rightarrow (2)$ of Theorem \ref{MainB} fails if we replace ``semistability'' with ``stability'', as stability of $E$ only implies semistability of $\wedge^n E$ for $n \ge 2$. The notion of cohomological stability, introduced in \cite{EL} and studied further in \cite{MS} and elsewhere, turns out to be more natural here. We recall the definition:

\begin{definition} A vector bundle $E \to C$ of rank $r$ is \textsl{cohomologically stable} (resp., \textsl{cohomologically semistable}) if for $1 \le n \le r-1$ we have $h^0 ( C, N^{-1} \otimes \wedge^n E ) = 0$ for all line bundles $N$ with $\deg ( N ) \ge n \cdot \mu ( E )$ (resp., $\deg (N) > n \cdot \mu (E)$). \end{definition}

It is not hard to see that $E$ is cohomologically stable if and only if $s_1 ( \wedge^n E ) > 0$ for $1 \le n \le r - 1$. 
\begin{comment}
Details: If $N$ is invertible then $N$ is a subsheaf of $F$ if and only if $h^0 ( N^{-1} \otimes F ) > 0$. Hence the defn of cohom stab says that for all invertible subsheaves $N$ of $\wedge^n E$ we have $\deg (N) < n \cdot \mu (E) = \mu ( \wedge^n E )$. This is equivalent to
\[ \deg ( \wedge^n E ) - \rank (\wedge^n E ) \cdot \deg (N) \ > \ 0 , \]
that is, $s_1 ( \wedge^n E ) > 0$.
\end{comment}
 Thus if $\rank(E) = 2$ then cohomological stability is equivalent to slope stability, both being equivalent to the single inequality $s_1 ( E ) > 0$. For $r \ge 3$, however, cohomological stability is stronger. For example, by \cite[Proposition 3.1]{CH2} a generic symplectic bundle $E$ of rank $2n$ and trivial determinant is a stable vector bundle, but if $\rank (E) \ge 4$ then $E$ is only cohomologically semistable since $h^0 ( C, \wedge^2 E ) \ne 0$.

Cohomological stability lends itself to a characterisation via inflectional loci more naturally than slope stability. An argument virtually identical to that of Theorem \ref{MainB} shows the following:

\begin{theorem} \label{MainBmodified} Let $E$ be a bundle of rank $r$ and degree $d \le r(1 - 2g)$ over $C$. Then $E$ is cohomologically stable if and only if for all $M \in \Pic^0 ( C )$ and for $1 \le n \le r-1$, the osculating space $\Osc^k ( S_n, y ; \cL_{n, M} )$ are of dimension $k \cdot \binom{r}{n}$ for $0 \le k \le n \cdot \mu ( E^* ) - (2g - 1)$ and for all $y \in S_n$. \end{theorem}

\begin{comment}
\begin{proof}
Suppose $E$ is cohomologically stable, so $s_1 ( \wedge^n E ) > 0$ for $1 \le n \le r-1$. By Theorem \ref{MainA}, for all $M \in \Pic^0 (C)$ and all $y \in S_n$, the osculating space $\Osc^k ( S_n , y ; \cL_{n, M}) )$ has dimension $k \cdot \binom{r}{n}$ for all $k \ge 0$ such that
\[ \deg ( \wedge^n E ) + \rank ( \wedge^n E ) ( 2g-1 + k ) \ \le \ 0 , \]
%\[ \mu ( \wedge^n E ) + ( 2g-1 + k ) \ \le \ 0 , \]
%\[ n \cdot \mu ( E ) + (2g-1) + k \ \le \ 0 , \]
%\[ k \ \le \ n \cdot \mu ( E^* ) - ( 2g-1 ) \]
that is, $0 \le k \le n \cdot \mu ( E^* ) - (2g - 1)$.

Conversely, assume that for all $M \in \Pic^0 ( C )$ and for $1 \le n \le r-1$, the osculating spaces $\Osc^k ( S_n, y ; \cL_{n, M} )$ are of dimension $k \cdot \binom{r}{n}$ for $0 \le k \le n \cdot \mu ( E^* ) - (2g - 1)$. By Theorem \ref{MainA}, for $1 \le n \le r-1$ we have
\begin{multline*} s_1 ( \wedge^n E ) \ > \ \deg ( \wedge^n E ) + \rank ( \wedge^n E ) \cdot ( 2g - 1 + (n \cdot \mu ( E^* ) - (2g - 1)) ) \\
 > \ \deg ( \wedge^n E ) + \rank ( \wedge^n E ) \cdot \mu ( \wedge^n E^* ) \ = \ 0 . \end{multline*} \end{proof}
\end{comment}

\section{Inflectional loci of general scrolls} \label{dimInfl}

In this section we will give another application of the ideas in {\S} \ref{SectionQuotInfl}, showing that for general $S$ and $M$ the inflectional loci $\Phi^k (\cL_M)$ are of the expected dimension. Firstly, we recall Kleiman's theorem on transversality of intersection of translates \cite[Theorem 2]{Kle}. Note that this theorem requires the hypothesis $\mathrm{char}(\K) = 0$.

\begin{theorem} \label{Kleiman} Let $G$ be a connected algebraic group (not necessarily linear). Let $Z$ be an irreducible variety with a transitive $G$-action. Let $a \colon X \to Z$ and $b \colon Y \to Z$ be maps of nonsingular integral schemes. For $\gamma \in G$, let $\gamma \cdot Y$ denote $Y$ considered as a $Z$-scheme by the map $y \mapsto \gamma \cdot b (y)$. Then there exists a dense open subset $U \subseteq G$ such that for $\gamma \in U$, the fibre product $(\gamma \cdot Y) \times_Z X$ is either empty, or equidimensional and smooth of the expected dimension $\dim (X) + \dim (Y) - \dim (Z)$. \end{theorem}

We will also require the following generalisation of \cite[Lemma 3.3]{LNew}. Recall that a vector bundle map $E \to \Kc \otimes E$ can be composed with itself $m$ times to obtain a map $E \to \Kc^m \otimes E$. Thus it makes sense to speak of \emph{nilpotent} maps $E \to \Kc \otimes E$.

\begin{lemma} Let $E$ be a vector bundle. Assume that there is no rank one nilpotent map $E \to \Kc \otimes E$. Then for any line bundle $L$ and all integers $\ell$, every component of the Quot scheme $Q_\ell ( L \otimes E )$ is smooth and of the expected dimension at all points. \label{QuotSmooth} \end{lemma}

\begin{proof} The association $\left[ N \to L \otimes E \right] \ \mapsto \ \left[ L^{-1}N \otimes E \right]$ defines an isomorphism $Q_\ell ( L \otimes E ) \isom Q_{\ell - \deg (L)} ( E )$. Thus it suffices to prove the statement for $L = \Oc$. By the theory of Quot schemes, we must show that the obstruction space $H^1 ( C, \Hom ( N, E/N ))$ is zero for all invertible subsheaves $N \subset E$. Write $\bN$ for the saturation of $N$ and $T \cong \bN/N$ for the torsion subsheaf of $E/N$. From the exact sequence
\[ 0 \ \to \ \Hom ( N, T ) \ \to \ \Hom (N, E/N) \ \to \ \Hom ( N, E/\bN ) \ \to \ 0 \]
and since $\Supp (T)$ is zero-dimensional, we have
\begin{equation} H^1 ( C, \Hom (N, E/N) \ \cong \ H^1 ( C, \Hom ( N, E/\bN ) ) . \label{LNewOne} \end{equation}
Similarly, in view of the exact sequence
\[ 0 \ \to \ \Hom ( \bN, E/ \bN ) \ \to \ \Hom (N, E / \bN)) \ \to \ {\cE}xt^1 ( T, E / \bN ) \ \to \ 0 , \]
there is a surjection $H^1 ( C, \Hom ( \bN, E / \bN )) \twoheadrightarrow H^1 ( C, \Hom ( N, E/\bN ) )$. 
%\begin{equation} H^1 ( C, \Hom ( \bN, E / \bN )) \ \twoheadrightarrow \ H^1 ( C, \Hom ( N, E/\bN ) ) . \label{LNewTwo} \end{equation}
Combining this with (\ref{LNewOne}), %and (\ref{LNewTwo}), 
 it suffices to prove that $h^1 ( C, \Hom ( \bN, E/\bN )) = 0$. This follows from the hypothesis on $E$ by the argument of \cite[Lemma 3.3]{LNew}. \end{proof}

\begin{remark} By \cite{Lau}, a general bundle $E$ admits no nilpotent map $E \to \Kc \otimes E$, so the hypothesis of Lemma \ref{QuotSmooth} and the following theorem is satisfied if $E$ is general in moduli (with no assumptions on the smooth curve $C$). %Clearly it then holds also for $L \otimes E$, for any line bundle $L$ over $C$.
 \end{remark}

Now fix $d < r(1-g)$ and write $n := r(1-g) - d - 1$. As in (\ref{kOne}), set
\[ k' \ := \ \max \{ k \ge 0 : kr \le n \} . \]

\begin{theorem} \label{MainC} Let $E$ be a bundle of rank $r$ and degree $d$. Assume that there is no rank one nilpotent map $E \to \Kc \otimes E$. Let $\pi \colon S := \PP E \to C$ be the associated projective bundle, and $\cL_M$ the line bundle $\cO_{\PP E} (1) \otimes \pi^* M$ as before. For general $M \in \Pic^0 ( C )$, the following holds.
\begin{enumerate}
\item[(a)] The linear series $|\cL_M|$ has dimension $n$.
\item[(b)] If $0 \le k < k'$, then $d^k ( \cL_M ) = kr$ and $\Phi^k ( \cL_M )$ is empty.
\item[(c)] We have $d^{k'} (\cL_M) = k'r$ and $\Phi^{k'} ( \cL_M )$ is either empty or of the expected dimension $(k' + 1)r - n - 1$.
\end{enumerate} \end{theorem}

\begin{proof} (a) By Serre duality, it suffices to show that
\begin{equation} h^1 ( C , E^* \otimes M ) \ = \ h^0 ( C, \Kc M^{-1} \otimes E ) \ = \ 0 \hbox{ for general } M \in \Pic^0 ( C ) . \label{Mvanish} \end{equation}
If $M$ is any degree zero invertible subsheaf of $\Kc \otimes E$, then by hypothesis and by Lemma \ref{QuotSmooth}, we have
%\begin{multline*} \dim ( Q_0 ( \Kc \otimes E ) ) \ = \ \chi ( C, \Hom ( M , (\Kc \otimes E)/M ) ) \ = \ d + 2r(g-1) - (r-1)(g-1) \\
% = \ d + r(g-1) + (g-1) \ < \ (g-1) , \end{multline*}
\[ \dim ( Q_0 ( \Kc \otimes E ) ) \ = \ \chi ( C, \Hom ( M , (\Kc \otimes E)/M ) ) = \ d + r(g-1) + (g-1) . \]
As $d < r(1-g)$, this is at most $g-2$. Thus a general $M \in \Pic^0 ( C )$ cannot be a subsheaf of $\Kc \otimes E$. Statement (\ref{Mvanish}) follows.

(b) The same argument as in part (a) shows that for any $M \in \Pic^0 (C)$ we have
\[ \dim ( Q_{-(k+1)} (\Kc M^{-1} \otimes E ) ) %\ = \ (r-1)(k+1) + d + r (2g-2) - (-(k+1)) - (r-1)(g-1)
 \ = \ r(k+1) + d + (r+1)(g-1) . \]
%(Note that this Quot scheme may be reducible.) 
Next, recall the space $R^k_M$ and the map $b$ from (\ref{defnRkM}). For any $M \in \Pic^0 ( C )$, we have
\[ b \left( Q_{-(k+1)} ( \Kc M^{-1} \otimes E ) \right) \ = \ M^{-1} \cdot b \left( Q_{-(k+1)} ( \Kc \otimes E ) \right) \]
where $\Pic^0 ( C )$ acts on itself by translation. Thus for general $M$, by Theorem \ref{Kleiman} each component of 
%\[ R^k_M \ = \ Q_{-(k+1)} ( \Kc M^{-1} \otimes E ) \times_{\Pic^0 (C)} C \]
 $R^k_M$ is empty or has dimension
\begin{multline} \dim ( Q_{-(k+1)} ( \Kc M^{-1} \otimes E ) ) + \dim (C) - \dim ( \Pic^0 (C) ) \\
 = \ r(k+1) + d + (r+1)(g-1) + 1 - g 
% = \ r(k+1) -(-d - r(g-1) ) + (g - 1) + 1 - g \\
% = \ r(k+1) - (n+1)
\ = \ r(k+1) - n - 1 \label{dimRkM} \end{multline}
when this is nonnegative. For $k < k'$, this is negative. Taking $M$ general enough (bearing in mind that $Q_{-(k+1)} ( \Kc \otimes E )$ may not be irreducible), we may assume that $R^k_M$ is empty. By Proposition \ref{ParamInfl} (b) we conclude that $d^k (\cL_M) = kr$ and $\Phi^k ( \cL_M )$ is empty for $0 \le k < k'$. %, so we obtain statement (b).

(c) By part (b), for all $x \in S$ and for $0 \le k < k'$ the space $\Osc^k ( S, x ; \cL_M )$ has dimension $kr$. Therefore, by Proposition \ref{ParamInfl} (b) the locus of $x \in S$ where $\dim ( \Osc^{k'} ( S, x ; \cL_M) ) < k'r$ is exactly the image of $e_{k'} \colon R^{k'}_M \dashrightarrow S$. Thus it has dimension at most $\dim ( R^{k'}_M ) = (k'+1)r - n - 1$ (cf.\ (\ref{dimRkM})). As this is at most $r - 1$, we have $d^{k'} (\cL_M) = k'r$. Since $\Phi^{k'} ( \cL_M )$ is determinantal, if it is nonempty then in this case it has dimension exactly the expected one $(k'+1)r - n - 1$. This completes the proof of (c). \end{proof}

\begin{remark} If $d \le r(1-2g)$, then for some values of $k$, we can strengthen Theorem \ref{MainC} slightly using the results of the previous section. Let $E$ be as in Lemma \ref{QuotSmooth}. Then for any $\ell$ such that $d - r \ell - (r-1)(g-1) < 0$, the Quot scheme $Q_\ell ( E )$ is empty. Thus 
%\[ s_1 ( E ) \ = \ \min\{ d - r \cdot \deg (N) : N \hbox{ an invertible subsheaf of } E \} \ \ge \ (r-1)(g-1) . \]
 $s_1 ( E ) \ge (r-1)(g-1)$. It follows easily that $s_1 (E)$ is the maximal value $(r-1)(g-1) + \delta$ defined in Theorem \ref{HirschBound}. By Corollary \ref{SpecialCases} (c), for
\[ 0 \ \le \ k \ < \mu ( E^* ) - \frac{(r+1)g - 1}{r} \ = \ -\frac{d}{r} - g - \frac{g-1}{r} , \]
we have $d^k ( \cL_M) = kr$ and $\Phi^k ( \cL_M )$ is empty for \emph{all} $M \in \Pic^0 (C)$ (not only for general $M$). For comparison,
\[ k' \ = \ \left\lfloor -\frac{d}{r} - (g-1) - \frac{1}{r} \right\rfloor . \]
\end{remark}

\subsection{Scrolls which are not linearly normal} \label{ScrollsIncSyst}

Using the general result proven in the appendix, we can generalise Theorem \ref{MainC} to incomplete linear systems. During this subsection, we assume $S = \PP E$ and $M \in \Pic^0 ( C )$ are fixed, and are general in the sense of Theorem \ref{MainC}. We now change our notation.

For any nonzero subspace $W \subset H^0 ( S, \cL_M )$, we consider the natural map $\psi_W \colon S \dashrightarrow \PP W^*$, which is the composition of $S \dashrightarrow | \cL_M |^*$ with the projection to $\PP W^*$. For $x \in S$ and $k \ge 0$ we have the osculating space $\Osc^k_W (S, x) \subseteq \PP W^*$. Write $d^k_W$ for the dimension of $\Osc^k_W ( S, x )$ at a general point $x \in S$, and $\Phi^k_W$ for the associated inflectional locus. For $0 < m < n$, write $k'_m := \max \{ k \ge 0 : kr \le m \}$.

\begin{theorem} \label{IncompleteSystemsScrolls} Let $E$, $S$ and $M$ be fixed as above. Let $W \subseteq H^0 ( S, \cL_M )$ be a general subspace of dimension $m + 1 < n + 1$.
\begin{enumerate}
\item[(a)] For $0 \le k < k'_m$, we have $d^k_W = kr$ and $\Phi^k_W$ is empty.
\item[(b)] Suppose $k = k'_m$. Then $d^{k'_m}_W = k'_m r$ and the inflectional locus $\Phi^{k'_m}_W$ is either empty or equidimensional of the expected dimension $(k'_m + 1)r - m - 1$. If $k'_m \ne k'$ then $\Phi^{k'_m}_W$ is smooth. If $k'_m = k'$ then $\Phi^{k'_m}_W$ is smooth except possibly along the closed sublocus $\Phi^{k'}$. \end{enumerate} \end{theorem}

\begin{proof} By hypothesis and by Theorem \ref{MainC} we have $d^k = kr$ for $0 \le k \le k'$. In particular $d^{(k'_m - 1)} = d^{k'_m} - r$ and hypothesis (\ref{TechDimAssp}) holds. 
%\[ d_{k'_m - 1} \ = \ ( k'_m - 1 ) r \ = \ k'_m r - r \ = \ d_{k'_m} - r \]
 Suppose $0 \le k < k'_m$. Since $W$ is general, by Theorem \ref{IncompleteSystems} (a) we have $d^k_W = d^k = kr$ and also $\Phi^k_W = \Phi^k$. But the latter is empty by Theorem \ref{MainC}. This proves (a).

 If $k = k'_m$ then, by Theorem \ref{IncompleteSystems} (b) we have $d^{k'_m}_W = d^{k'_m} = k'_m r$, and the locus $\Phi^k_W$ is the union of $\Phi^k$ and a locus which, if nonempty, is smooth and equidimensional of the expected dimension $r + d^{k'_m} - m - 1 = (k'_m + 1)r - m - 1$. If $k'_m < k'$ then $\Phi^{k'_m}$ is empty by Theorem \ref{MainC}, so $\Phi^{k'_m}_W$ is smooth. Suppose $k'_m = k'$. As $\Phi^{k'_m}_W$ is determinantal, every component has dimension at least
\[ (k'_m + 1)r - m - 1 \ = \ (k' + 1)r - m - 1 \ > \ (k' + 1)r - n - 1 . \]
Hence if $\Phi^{k'}$ is nonempty then it belongs to the closure of the aforementioned smooth equidimensional locus. \end{proof} %\end{comment}

\begin{remark} In \cite[Theorem 2 and Corollary 1]{LMP}, a formula is given for the cohomology class of $\Phi^{k'_m}$, assuming this is of the expected dimension $(k'_m + 1)r-m-1$. In particular, when $m = (k'_m + 1) r - 1$, 
%so the expected dimension $(k'_m + 1)r - m - 1$ is equal to $(k'_m + 1)r - (k'_m + 1)r + 1 - 1 = 0$,
 the formula enumerates the finitely many inflection points of order $k'_m$. Theorem \ref{MainC} and Corollary \ref{IncompleteSystemsScrolls} show that the assumption of expected dimension in \cite{LMP} is satisfied when $E$ and $M$ are chosen generally and when the projection is general. \end{remark}

\appendix

\section{Inflectional loci under general projections} \label{SectionIncSyst} 

In this section we will prove a general statement on the behaviour of inflectional loci under general projections. Let $X$ be a smooth projective variety of dimension $r$. Let $\psi \colon X \dashrightarrow \PP V^* = \PP^n$ be a map, where $n \ge 2$. We denote by $d^k$ the dimension of $\Osc^k ( X, x)$ at the generic point, and write $k' := \max\{ k \ge 0 : d^k \le n \}$. For $0 \le k \le k'$, we consider the associated inflectional loci $\Phi^k \subset X$.

Now suppose $W \subset V$ is a proper subspace of dimension $m + 1$. Composing with the projection $\PP V^* \dashrightarrow \PP W^*$, we obtain a map $X \dashrightarrow \PP^m$, together with osculating spaces $\Osc^k_W (X, x) \subseteq \PP^m$ and inflectional loci $\Phi^k_W \subset X$. We write $d^k_W$ for the generic value of $\dim ( \Osc^k_W ( X, x) )$. Write $k'_m := \max \{ k : d^k \le m \}$.

Note moreover that the centre of the projection $\PP^n \dashrightarrow \PP^m$ is $\PP W^\perp = \PP^{n-m-1}$, where $W^\perp = \Ker ( V^* \to W^* )$ is the orthogonal complement of $W$ under the natural pairing $V \times V^* \to \K$. 

\begin{theorem} \label{IncompleteSystems} Let $\psi \colon X \dashrightarrow \PP V^* = \PP^n$ be as above, and let $W \subset V$ be a proper subspace of dimension $m + 1$. Assume in addition that
\begin{equation} d^{(k'_m - 1)} \ \le \ d^{k'_m} - r , \label{TechDimAssp} \end{equation}
where $\dim (X) = r$. If $W$ is general in the Grassmannian $\Gr ( m+1, V )$, the following holds.
\begin{enumerate}
\item[(a)] Suppose $0 \le k < k'_m$. Then $d^k_W = d^k$ and $\Phi^k_W = \Phi^k$. In particular, the projection $\Osc^k ( X, x ) \to \Osc^k_W ( X, x )$ is an isomorphism for all $x \not\in \Phi^k$.
\item[(b)] Suppose $k = k'_m$. Then $d^k_W = d^k$ and $\Phi^k_W$ is the union of $\Phi^k$ and a locus which, if nonempty, is smooth and equidimensional of the expected dimension $r + d^k_W - m - 1$.
\end{enumerate}
\end{theorem}

\begin{proof} For $k \ge 0$, it follows from the definitions that $\Osc^k_W (X, x)$ is the image of $\Osc^k (X, x)$ under the projection $\PP^n \dashrightarrow \PP^m$. For $0 \le k \le k'_m$, we have
\[ \dim ( \Osc^k ( X, x ) ) + \dim ( \PP W^\perp ) - n \ \le \ d^k + (n-m-1) - n \ = \ d^k - m - 1 \ < \ 0 . \]
Thus, since $W$ is general, $\Osc^k ( X, x ) \cap \PP W^\perp$ is empty for general $x \in X$. Therefore, for $0 \le k \le k'_m$ we have $d^k_W = d^k$. Hence
\begin{equation} \hbox{$x \in \Phi^k_W$ if and only if $x \in \Phi^k$ or $\Osc^k ( X, x ) \cap \PP W^\perp$ is nonempty.} \label{critinflprelim} \end{equation}
We reformulate this last property. For fixed $k$ with $0 \le k \le k'$, denote by $Z$ the Grassmann variety $\Gr ( d^k + 1 , V^* )$. For any $W \subset V$, we consider the special Schubert cycle
\[ Y_{W^\perp} \ := \ \{ \Lambda \subset V^* \hbox{ of dimension } d^k + 1 : \Lambda \cap W^\perp \ne 0 \} \ \subset \ Z . \]
The association $x \mapsto \Osc^k ( X, x )$ defines a rational map $a \colon X \dashrightarrow Z$, whose indeterminacy locus is exactly $\Phi^k$. (This may be thought of as a generalised Gauss map.) Write $X' := X \setminus \Phi^k$ for the locus of definition of $a$. Then (\ref{critinflprelim}) is equivalent to
\begin{equation} \Phi^k_W \ = \ \Phi^k \cup a^{-1} \left( Y_{W^\perp} \right) \ \cong \ \Phi^k \cup \left( Y_{W^\perp} \times_Z X' \right) . \label{critinfl} \end{equation}

Now $G := \GL ( V^* )$ has a natural transitive action on $\Gr ( \ell , V^* )$ for any $1 \le \ell \le n$. Unwinding the definitions, we see that
\begin{equation} \gamma \cdot Y_{W^\perp} \ = \ Y_{\gamma \cdot W^\perp} . \label{ActionsMatch} \end{equation}
Let us estimate the dimension of $Y_{W^\perp}$. Any $\Lambda \in Y_{W^\perp}$ fits into an exact diagram
\[ \xymatrix{ 0 \ar[r] & \lambda \ar[r] \ar[d]^= & \Lambda \ar[r] \ar[d] & \Lambda / \lambda \ar[r] \ar[d] & 0 \\
 0 \ar[r] & \lambda \ar[r] & V^* \ar[r] & V^* / \lambda \ar[r] & 0 } \]
where $\lambda$ is a one-dimensional subspace of $W^\perp$. Thus
\[ \dim ( Y_{W^\perp} ) \ \le \ \dim ( \PP W^\perp ) + \dim ( \Gr ( d^k , n ) ) \ = \ (n-m-1) + d^k (n-d^k ) . \]
Computing, we obtain
\begin{equation} \dim ( Y_{W^\perp} ) + \dim ( X ) - \dim ( Z ) %\ \le \ (n-m-1) + d^k ( n - d^k ) + r - (d^k + 1)(n - d^k) \ = \ (n-m-1) + r - n + d^k 
 \ \le \ r + d^k - m - 1 . \label{IntersectionDimension} \end{equation}

Suppose now that $k < k'_m$. By (\ref{TechDimAssp}), we have $d^k \le d_{k'_m} - r \le m - r$, from which it follows that %\[ r + d^k - m - 1 \ \le \ r + m - r - m - 1 = -1 \]
 $r + d^k - m - 1 < 0$. Perturbing $W^\perp$ to a general translate $\gamma \cdot W^\perp$ if necessary, by Theorem \ref{Kleiman} and (\ref{ActionsMatch}) we may assume $Y_{W^\perp} \times_Z X'$ is empty. By (\ref{critinfl}), we conclude that $\Phi^k_W = \Phi^k$. Statement (a) now follows.

Next, suppose $k = k'_m$. Again, possibly after perturbing $W^\perp$, by Theorem \ref{Kleiman} and (\ref{ActionsMatch}) we may assume that $\left( \gamma \cdot Y_{W^\perp} \right) \times_Z X'$ is either empty, or smooth and equidimensional of dimension at most $r + d^k - m - 1$. In view of (\ref{critinfl}), and since $\Phi^k_W$ is determinantal, if $\left( \gamma \cdot Y_{W^\perp} \right) \times_Z X'$ is nonempty then the dimension is exactly $r + d^k - m - 1$. This proves part (b). \end{proof}

\begin{remark} If $X$ is a curve, part (a) follows from \cite[Proposition 4.2]{Pi1977}. Compare also with \cite[Theorem 4.1]{Pi1978} on the behaviour of \emph{polar loci} under general projections. \end{remark}

\bibliography{inflection_Segre_biblio}{}

\bibliographystyle{alpha}

\end{document}